\documentclass{TPmod}

\usepackage{tikz-cd}
\usepackage{enumitem}
\usepackage{mathtools}
\usepackage{verbatim}
\usepackage{subcaption}
\usepackage{todonotes}
\usepackage[citestyle=alphabetic,bibstyle=alphabetic,backend=bibtex, maxbibnames=10,minbibnames=5]{biblatex}
\newtheorem{ass}{Assumption}[section]

\numberwithin{equation}{section}

\newtheorem{conjmain}[mainthm]{Conjecture}

\newtheorem*{question}{Question}

\crefname{thm}{Theorem}{Theorems}
\crefname{prop}{Proposition}{Propositions}

\def\hookar{\ar@{^{(}->}}

\def\eps{\epsilon}

\def\mir0{F_0}

\newcommand{\NN}{\mathbb{N}}

\newcommand{\harpoon}{\vec}

\def\and{\, \& \,}

\renewcommand{\vec}[1]{\mathbf{#1}}

\def\Hom{\mathrm{Hom}}

\newcommand{\ZZ}{\mathbb Z}

\def\comb{\mathsf{comb}}
\def\V{\mathbf{V}}
\title{On the positivity and integrality of coefficients of mirror maps}

\author{Sophie Bleau and Nick Sheridan}

\begin{abstract} 
We present natural conjectural generalizations of the `positivity and integrality of mirror maps' phenomenon, encompassing the mirror maps appearing in the Batyrev--Borisov construction of mirror Calabi--Yau complete intersections in Fano toric varieties as a special case. 
We find that, given the combinatorial data from which one constructs a mirror pair of Calabi--Yau complete intersections, there are \emph{two} ways of writing down an associated `mirror map': one which is the `true mirror map', meaning the one which appears in mirror symmetry theorems; and one which is the `naive mirror map'. 
The two are equal under a certain combinatorial criterion which holds e.g. for the quintic threefold, but not in general. 
We conjecture (based on substantial computer checks, together with proofs under extra hypotheses) that the naive mirror map always has positive integer coefficients, while the true mirror map always has integer (but not necessarily positive) coefficients. 
Most previous works on the integrality of mirror maps concern the naive mirror map, and in particular, only apply to the true mirror map under the combinatorial criterion mentioned above.
\end{abstract}

\begin{document}
\pagenumbering{roman}
\maketitle

\pagenumbering{arabic}

\section{Introduction}

\subsection{Mirror symmetry context}

In this subsection, we explain the mirror symmetry context for our work. 
The reader unfamiliar with mirror symmetry, Gromov--Witten invariants, or Yukawa couplings is reassured that they will not be mentioned outside of this subsection, and referred to the excellent \cite{cox1999mirror} if they would like to learn. 

Genus-zero enumerative mirror symmetry for the quintic threefold is a relationship between, on the one hand, the generating function 
for genus-zero Gromov--Witten invariants of the quintic threefold; and on the other, the Yukawa coupling for the mirror quintic family. 
Explicitly, let $X \subset \mathbb{CP}^4$ be a smooth quintic hypersurface, and $N_d$ the genus-zero, $3$-point Gromov--Witten invariant of degree-$d$ curves in $X$, with the hyperplane class $H \in H^2(X)$ inserted at all three points, $\mathrm{GW}_{0,3}^{X,d}(H,H,H)$. 
Then the relevant generating function for Gromov--Witten invariants is 
$$ f(Q) = \sum_{d \ge 0} N_d \cdot Q^d.
$$
On the other side, let $Y_q$ be a crepant resolution of the quotient of the hypersurface
$$ \left\{ y_1y_2y_3y_4y_5 = q \cdot \left(y_1^5+y_2^5+y_3^5+y_4^5+y_5^5\right) \right\} \subset \mathbb{CP}^4$$
by the group $\Gamma = \ker((\Z/5\Z)^3 \xrightarrow{\sum} \Z/5\Z)/(\Z/5\Z)$ acting diagonally on $\mathbb{CP}^4$ by fifth roots of unity. 
The Yukawa coupling for this family is
$$ g(q) = \int_{Y_q} \Omega_q \wedge \nabla_{q \partial_q} \nabla_{q \partial_q} \nabla_{q \partial_q} \Omega_q$$
where $\Omega_q \in \Omega^{3,0}(Y_q)$ is the family of `normalized' holomorphic volume forms, and $\nabla$ is the Gauss--Manin connection. 

Genus-zero enumerative mirror symmetry for the quintic (first conjectured by Candelas--de la Ossa--Green--Parkes \cite{candelas1991pair} and proved by Givental \cite{Givental_quintic} and Lian--Liu--Yau \cite{LLY}) then says
$$g(q) = f(Q(q))$$
where $Q(q)$ is an explicit power series whose form we give in \eqref{eq:quint_mm} below. 
One derives explicit formulae for the Yukawa coupling $g(q)$ and the mirror map $Q(q)$ by solving the Picard--Fuchs equation associated to the family $Y_q$, which can be done as the latter is of hypergeometric type (see \cite[Section 6.3.4]{cox1999mirror} for an exposition). 
The most striking consequence of this version of mirror symmetry is that we can solve for the generating function $f(Q)$ for the Gromov--Witten invariants $N_d$, thus giving explicit formulae for the latter. 

The mirror symmetry conjecture was generalized to Calabi--Yau complete intersections in Fano toric varieties by Batyrev--Borisov \cite{Batyrev_Borisov}, and genus-zero enumerative mirror symmetry was proved in this context by Givental \cite{Givental_toric}. 
In general, the Gromov--Witten generating function $f(\vec{Q})$ associated to a complete intersection $X$, and the Yukawa coupling $g(\vec{q})$ associated to the mirror family of complete intersections $Y_{\vec{q}}$, depend on multiple parameters $\vec{Q} = (Q_i)_{i=1}^N$, respectively $\vec{q} = (q_i)_{i=1}^N$. 
The number of parameters $Q_i$ is equal to the rank of $H_2(X)$, while the number of parameters $q_i$ is the dimension of the moduli space of complex deformations of $Y_{\vec{q}}$. 
It is a non-trivial feature of Batyrev--Borisov's construction that these numbers of variables coincide; e.g., in the case of the quintic they are both $1$.  
There are now $N$ mirror maps, each with $N$ variables: $\vec{Q}(\vec{q}) = (Q_i(q_1,\ldots,q_N))_{i=1}^N$.

The main players in the present work are the mirror maps $\vec{Q}(\vec{q})$.  
In the case of the quintic the mirror map is given by 
\begin{align}
\label{eq:quint_mm}    Q_{\text{quintic}}(q) &= q^5 \cdot \exp\left(5\phi_1\left(q^5\right)/\phi_0\left(q^5\right)\right) \quad \text{where}\\
\label{eq:quint_phi0}    \phi_0(z) &= \sum_{k \ge 0} \frac{(5k)!}{(k!)^5} \cdot z^k,\\
\label{eq:quint_phi1}    \phi_1(z) &= \sum_{k \ge 1} \frac{(5k)!}{(k!)^5} \cdot \sum_{j=k+1}^{5k} \frac{1}{j} \cdot z^k.
\end{align}

It was observed in the early days of mirror symmetry that the coefficients of $Q_{\text{quintic}}(q)$ seemed all to be positive integers (the earliest references for integrality we can find are \cite{batyrev1995generalized,Lian1996}; the first mention we can find of positivity is in \cite{krattenthaler2011analytic}, although it had surely been remarked before then). 
In fact, the fifth root of $Q_{\text{quintic}}$ is even integral:
$$Q_{\text{quintic}}^{1/5} = z \cdot \exp\left(\frac{\phi_1(z)}{\phi_0(z)}\right) = z+154z^2 + 155423z^3 + 237738254z^4 + 439875902939z^5 +  \ldots, $$
and its logarithm even has positive coefficients:
$$\frac{\phi_1(z)}{\phi_0(z)} = 154z + 143565z^2 + \frac{645061600}{3} z^3 + \frac{789462914125}{2} z^4 +\ldots.$$
The integrality of the coefficients of $Q_{\text{quintic}}$ was first proved by Lian--Yau \cite{Lian1998}, who also proved integrality of the coefficients of $Q_{\text{quintic}}^{1/5}$ \cite{Lian-Yau-nthroot} (a different approach was later developed by Kontsevich--Schwarz--Vologodsky \cite{KSV}; the integrality question has also been considered from a physics perspective, e.g. in \cite{JockersMayr}). 
The positivity of the coefficients of $\phi_1/\phi_0$ was first proved by Krattenthaler--Rivoal \cite{krattenthaler2011analytic}, who also address the question of its convergence.

\subsection{The conjectures}

We introduce notation:
\begin{align*}
    H(n) &\coloneqq \sum_{i=1}^n \frac{1}{i} \qquad \text{if $n \ge 0$ (we define $H(0) = 0$);}\\
    \comb(k_1,\ldots,k_m) &\coloneqq \frac{\left(\sum_{j=1}^m k_j\right)!}{\prod_{j=1}^m k_j!} \qquad \text{if all $k_j \ge 0$.}
\end{align*}

Now we introduce the data from which our power series are constructed. 
Let $(\vec{v}_{ij})$ be vectors in $\Z^d$ indexed by the set
\begin{equation}
    \label{eq:I}
    I = \{(i,j): 1 \le i \le p; 1 \le j \le q_i\}.
\end{equation}
When $p=1$, we will simply write $\vec{v}_j$ instead of $\vec{v}_{1j}$.

Define the linear map 
\begin{align*}
   \V: \Z^I & \to \Z^d \quad \text{sending}\\
    \vec{e}_{ij} & \mapsto \vec{v}_{ij}
\end{align*}
where $\vec{e}_{ij}$ is the basis vector of $\Z^I$ corresponding to $(i,j) \in I$. Let $K \subset \Z^I$ be its kernel. 
Define the monoid
\begin{align*}
    K_0 & = \{\vec{k} \in K: k_{ij} \ge 0 \text{ for all $(i,j) \in I$}\}   
\end{align*}
and the corresponding completed monoid ring
\begin{align*}
    \Q[[\vec{z}]]_{K_0} := \left\{ \sum_{\vec{k} \in K_0} c_{\vec{k}} \cdot \vec{z}^{\vec{k}}: c_{\vec{k}} \in \Q\right\}.
\end{align*}
It is the completion of the monoid ring $\Q[K_0]$ at the ideal consisting of sums with $c_{\vec{0}} = 0$. 

Define elements
$$\phi_0 = \sum_{\vec{k} \in K_0} \left(\prod_{i=1}^p \comb(k_{i1},\dots,k_{iq_i})\right) \cdot \vec{z}^{\vec{k}} \in \Z[[\vec{z}]]_{K_0}$$
and
$$\phi_{ij} = \sum_{\vec{k} \in K_0} \left(\prod_{r=1}^p \comb(k_{r1},\dots,k_{rq_r})\right) \cdot \left(H\left(\sum_{t=1}^{q_i} k_{it}\right) - H(k_{ij}) \right) \cdot \vec{z}^{\vec{k}} \in \Q[[\vec{z}]]_{K_0}.$$
We define the \emph{naive mirror map} to have components $z_{ij} \cdot \psi_{ij}^{\text{n}}$, where
$$\psi_{ij}^{\text{n}}(\vec{z}) = \exp(\phi_{ij}/\phi_0).$$

\begin{rmk}
    It may be helpful to note that $\phi_0$ consists precisely of the summands of the sum
    $$\prod_{i=1}^p \frac{1}{1-\sum_{j=1}^{q_i} z_{ij}} = \sum_{\vec{k} \in \ZZ_{\ge 0}^I} \left(\prod_{i=1}^p \comb(k_{i1},\dots,k_{iq_i})\right) \cdot \vec{z}^{\vec{k}}$$
    such that $\vec{k} \in K$.
\end{rmk}

We now introduce the hypothesis under which we conjecture that the naive mirror map has positive integer coefficients. 
Let $\Delta_i$ be the convex hull of the vectors $\vec{v}_{ij}$, together with the origin $\vec{0}$, for $i=1,\ldots,p$. 
Let $\Delta = \sum_{i=1}^p \Delta_i$ denote the Minkowski sum of the $\Delta_i$. 

\begin{ass}\label{ass}
We will always assume that the origin $\vec{0}$ lies in the interior of $\Delta$ (this is equivalent to assuming that $K \cap \NN_{>0}^I \neq \emptyset$), and that the vectors $\vec{v}_{ij}$ span $\Z^d$.   
\end{ass}

\begin{rmk}
Assumption \ref{ass} does not cause any loss of generality: if it is not satisfied, we can modify our data by restricting to the linear subspace supporting the face of $\Delta$ containing the origin in its interior, removing all vectors $\vec{v}_{ij}$ which are not contained in this subspace, and replacing $\Z^d$ with the lattice spanned by the remaining $\vec{v}_{ij}$; we will then arrive at an equivalent formula for the naive mirror map. 
\end{rmk}

We say that $\Delta$ is \emph{Fano} if the origin is the \emph{unique} interior lattice point of $\Delta$ (following \cite[Definition 2.1]{Kasprzyk_Fano}); and we call the data $(\vec{v}_{ij})$ Fano in this case.

\begin{rmk}
    We say that the data $(\vec{v}_{ij})$ \emph{arise from a nef partition} if the vectors $\vec{v}_{ij}$ are the vertices of a reflexive polytope (without repetition); and the polytope $\Delta$ is reflexive \cite{Batyrev_Borisov}. Note that reflexive polytopes are necessarily Fano, but not vice-versa.
\end{rmk}

\begin{example}
    For an example of Fano data arising from a nef partition with $p>1$, one could consider $\vec{v}_{11} = (-1,0)$, $\vec{v}_{12} = (0,-1)$, $\vec{v}_{21} = (1,1)$ in $\Z^2$. 
    Then $\Delta$ is the convex hull of the vectors $\{(\pm 1,0),(0,\pm 1),(1,1)\}$, which is evidently reflexive and in particular Fano.
\end{example}

\begin{example}
    For an example of Fano data with $p>1$ such that $\Delta$ is not reflexive, we may take $\vec{v}_{11} = (-1,0,0)$, $\vec{v}_{12} = (0,-1,0)$, $\vec{v}_{13} = (0,0,-1)$, $\vec{v}_{21} = (1,1,2)$. 
    In particular, the polytope $\Delta$ is the convex hull of the vectors $\{(-1,0,0),(0,-1,0),(0,0,-1),(0,1,2),(1,0,2),(1,1,1),(1,1,2)\}$. To show that this polytope is Fano, we first observe that the origin is clearly an interior lattice point; we must show that it has no further interior lattice points. The projection of $\Delta$ to the $x$-$y$ plane is the polygon with vertices $\{(\pm 1,0),(0,\pm 1),(1,1)\}$, whose only interior lattice point is evidently the origin. Therefore, any interior lattice point of $\Delta$ must have the form $(0,0,z)$. As $\Delta$ lies within the region $\{-1 \le z \le 2\}$, the only possibility other than the origin is $z=1$. However this point lies on the boundary face of $\Delta$ with vertices $\{(-1,0,0),(0,-1,0),(1,0,2),(0,1,2)\}$; thus the origin is the only interior lattice point, and $\Delta$ is Fano. However, $\Delta$ is not reflexive, as the face with vertices $\{(1,0,2),(0,1,2),(1,1,2)\}$ corresponds to a non-integral vertex, $(0,0,-1/2)$, in the dual polytope.
\end{example}

\begin{conjmain}\label{conj:1}
If $(\vec{v}_{ij})$ are Fano, then for all $(i,j) \in I$:
\begin{enumerate}
    \item \label{it:1a} $\psi_{ij}^{\text{n}}$ has integer coefficients; i.e., it lies in $\Z[[\vec{z}]]_{K_0}$.
    \item \label{it:1b} $\log \psi^{\text{n}}_{ij} = \phi_{ij}/\phi_0$ has non-negative coefficients; i.e., it lies in $\Q_{\ge 0}[[\vec{z}]]_{K_0}$.
\end{enumerate}
Note that together, the two parts of the conjecture imply that $\psi_{ij}^{\text{n}} \in \mathbb{N}[[\vec{z}]]_{K_0}$.
\end{conjmain}

\begin{example}\label{eg:quintic}
    To recover the quintic map discussed in the first section, let \[(\vec{v}_j) = ((1,0,0,0),(0,1,0,0),(0,0,1,0),(0,0,0,1),(-1,-1,-1,-1)).\] 
    Then we have $K_0 = \{(k,k,k,k,k): k \in \NN\}$, and we find that $\phi_0(z)$ is given by the formula \eqref{eq:quint_phi0}, all $\phi_j(z)$ are equal and given by the formula \eqref{eq:quint_phi1}, and
    $$Q_{\text{quintic}}(q) = q^5 \cdot \psi^{\text{n}}_1\left(q^5\right)^5.$$
    Thus we find that Conjecture \ref{conj:1} says that $\psi^{\text{n}}_1 = Q_{\text{quintic}}^{1/5}/q$ has integer coefficients, and $\phi_1/\phi_0$ has positive coefficients, as remarked in the previous section. 
\end{example}

In order to formulate our second conjecture, concerning the true mirror map, we extend the definition of $\comb$:
$$\comb(k_1,\ldots,k_m) \coloneqq (-1)^{k_j+1} \frac{\left(\sum_{i=1}^m k_i\right)! (-k_j-1)!}{\prod_{\substack{1 \le i \le m\\ i \neq j}} k_i!},$$
defined if $k_i \ge 0$ for $i \neq j$, $k_j < 0$, and $\sum_{i=1}^m k_i \ge 0$. 
We define a new monoid:
$$K_{ij} = \left\{\vec{k} \in K: k_{lm} \ge 0 \text{ if and only if $(l,m) \neq (i,j)$; and } \sum_{l=1}^{q_i} k_{il} \ge 0\right\}.$$
We assume that $\vec{v}_{ij} \neq 0$, and define
$$\tau_{ij} = \sum_{\vec{k} \in K_{ij}} \left( \prod_{r=1}^p \comb(k_{r1},\dots,k_{rq_r}) \right) \cdot \vec{z}^{\vec{k}} \in \Q[[\vec{z}]]_{K_{ij}},$$
and we define the \emph{true mirror map} to be
$$\psi_{ij}^{\text{t}} = \exp((\phi_{ij} + \tau_{ij})/\phi_0) \in \Q[[\vec{z}]]_{K_0+K_{ij}}.$$

\begin{conjmain}\label{conj:2}
    If $(\vec{v}_{ij})$ are Fano, then $\psi_{ij}^{\text{t}}$ has integer coefficients for all $(i,j) \in I$.
\end{conjmain}

We remark that, in order for the infinite sum defining $\psi_{ij}^{\text{t}}$ to make sense, we need the following Lemma:

\begin{lem}[Cf. Lemma 9.2 of \cite{BeukersVlasenkoIII}] \label{lem:cone_conv}
If $(\vec{v}_{ij})$ are Fano, then the cone generated by $K_0$ and $K_{ij}$ is strictly convex.
\end{lem}
\begin{proof}
    It suffices to produce a vector $\vec{w}$ such that $\langle \vec{w},\vec{k} \rangle \ge 0$ for $\vec{k}$ in $K_0$ or $K_{ij}$, with equality if and only if $\vec{k} = 0$. We claim that setting $w_{ij} = 1/2$ and all other $w_{lm} = 1$ does the trick. 
    Indeed, if $\vec{k} \in K_0$, then it is clear that $\langle \vec{w},\vec{k} \rangle \ge 0$ with equality if and only if $\vec{k} = 0$. 
    If $\vec{k} \in K_{ij}$, then we have
    \begin{align*}
        0&=\sum_{(l,m) \in I} k_{lm} \vec{v}_{lm} \\
        \Rightarrow \vec{v}_{ij} &= \sum_{lm \neq ij} \frac{k_{lm}}{-k_{ij}} \vec{v}_{lm}.
    \end{align*}
    Now if $\sum_{lm \neq ij} \frac{k_{lm}}{-k_{ij}} <1$, then $\vec{v}_{ij}$ would be an additional interior lattice point of $\Delta$, which we have assumed to be non-zero, contradicting the assumption that $\vec{v}_{lm}$ are Fano. 
    Thus we have $\sum_{lm \neq ij} \frac{k_{lm}}{-k_{ij}} \ge 1$, and hence $\sum_{lm} k_{lm} - \frac{1}{2}k_{ij} > 0$ as required.
\end{proof}

The classical conjecture about the integrality of coefficients of mirror maps, see \cite[Conjecture 6.3.4]{batyrev1995generalized}, is equivalent to a special case of Conjecture \ref{conj:2} arising from a nef partition (see \cite[Section 6.3.4]{cox1999mirror} for an explanation of how to derive the formula, and \cite[Appendix C]{ganatra2024integrality} for the explicit derivation of the formula in the case $p=1$, see also \cite[Appendix]{BeukersVlasenkoIII} and \cite[Section 4]{adolphson2014}). 
Conjecture \ref{conj:2} is only equivalent to Conjecture \ref{conj:1} \eqref{it:1a} under an extra hypothesis:

\begin{lem}\label{lem:12equiv}
    If the origin does not lie in the interior of the convex hull of the vectors $\vec{v}_{\ell m}$ for $(\ell,m) \neq (i,j)$, and $-\vec{v}_{ij}$, then $K_{ij} = \emptyset$. In particular, $\tau_{ij} = 0$, so the naive mirror map is equal to the true mirror map, and Conjecture \ref{conj:1} \eqref{it:1a} is equivalent to Conjecture \ref{conj:2}.
\end{lem}
\begin{proof}
    We prove the contrapositive; so let us suppose that $K_{ij} \neq \emptyset$. 
    Then there exists a non-zero vector $\vec{k} \in K_{ij}$: so 
    $$\sum_{(\ell,m) \in I} k_{\ell m} \vec{v}_{\ell m} = \vec{0}$$
    where $k_{ij}<0$ and $k_{\ell m} \ge 0$ for $(\ell,m) \neq (i,j)$. 
    On the other hand, by Assumption \ref{ass}, there exists $\vec{k}' \in K \cap (\R_{>0})^I$. 
    Now let $\eps>0$ be sufficiently small that $k_{ij} + \eps \cdot k'_{ij} < 0$. 
    Then we have
    $$ \vec{0} = (-k_{ij} - \eps \cdot k'_{ij}) \cdot (-\vec{v}_{ij}) + \sum_{(\ell,m) \in I \setminus \{(i,j)\}} (k_{\ell m} +\eps \cdot k'_{\ell m}) \cdot \vec{v}_{\ell m},$$
    where the coefficients in front of $-\vec{v}_{ij}$ and $\vec{v}_{\ell m}$ are all strictly positive. 
    It follows that $\vec{0}$ lies in the interior of the convex hull of the vectors $\vec{v}_{\ell m}$ for $(\ell,m) \neq (i,j)$ and $-\vec{v}_{ij}$. 
\end{proof}

\begin{example}\label{eg:rank2}
    Let $(\vec{v}_j) = ((0,1),(1,1),(0,-1),(-1,1))$. Then we have an isomorphism
    \begin{align*}
        \Z^2 & \xrightarrow{\sim} K\\
        (a,b) & \mapsto (a,b,a+2b,b).
    \end{align*}
    under this isomorphism, $K_0 \cong \{(a,b):a \ge 0, b \ge 0\}$; $K_j = \emptyset$ for $j \neq 1$; and $K_1 = \{(a,b): a< 0, a+2b \ge 0\}$. 
    One may easily see that the coefficient of $z_1^{-2}z_2z_4$ in $\psi_1^{\text{t}}$ is equal to the coefficient of the same monomial in $\tau_1$, which is $-1$. 
    This shows both that $\psi_1^{\text{t}}$ has a negative coefficient, so the analogue of Conjecture \ref{conj:1} \eqref{it:1b} does not hold in this case; and also that $\psi_1^{\text{t}} \neq \psi_1^{\text{n}}$, as $(-2,1,0,1) \notin K_0$, so $z_1^{-2}z_2z_4 \notin \Q[[\vec{z}]]_{K_0}$.
\end{example}

\subsection{Known cases of Conjectures \ref{conj:1} and \ref{conj:2}}

The natural generalization of Example \ref{eg:quintic} is the case when $\vec{v}_1,\ldots,\vec{v}_d$ are a basis for $\Z^d$, and $\vec{v}_{d+1} = -\sum_{j=1}^d \vec{v}_j$. 
In this case we have $\psi_j^{\text{t}} = \psi_j^{\text{n}}$ by Lemma \ref{lem:12equiv}, so Conjectures \ref{conj:1} \eqref{it:1a} and \ref{conj:2} are equivalent; furthermore, all $\psi_j^{\text{t}}$ are equal by symmetry, so we will denote them all by $\psi$. 
The integrality part of our conjectures then says that $\psi$ should have integer coefficients. 
Lian--Yau proved that $\psi^d$ has integer coefficients when $d$ is prime \cite{Lian1998}; Zudilin extended this to the case that $d$ is a prime power \cite{Zudilin2002}; Lian--Yau proved that $\psi$ has integer coefficients when $d$ is prime \cite{Lian-Yau-nthroot}; and Krattenthaler--Rivoal proved this for general $d$, thus establishing Conjectures \ref{conj:1} \eqref{it:1a} and \ref{conj:2} in this case \cite{krattenthaler2010integrality}. 
Conjecture \ref{conj:1} \eqref{it:1b} was proved in this case by Krattenthaler--Rivoal \cite{krattenthaler2011analytic}.

Krattenthaler--Rivoal's integrality result covered a broad class of single-variable mirror maps (i.e., examples in which the rank of $K$ is $1$), which was subsequently enlarged by Delaygue \cite{Delaygue_single_variable}. 
The first results concerning integrality of multivariate mirror maps were obtained by Krattenthaler--Rivoal \cite{Krattenthaler_Rivoal_multivariate}, and subsequently generalized by Delaygue \cite{delaygue2013criterion}.

Delaygue gives a criterion for the integrality of mirror maps, which we show under certain hypotheses to be (non-obviously) equivalent to the condition that $\vec{v}_{ij}$ are Fano. 
As a result, we obtain the following result (substantially due to Delaygue, modulo our proof in Section \ref{sec:equivalence of Fano and Delaygue} of the equivalence of his criterion with our Fano hypothesis):

\begin{thm}\label{thm:conj1a}
    Suppose that we have an isomorphism of monoids, $K_0 \cong \NN^r$. Then Conjecture \ref{conj:1} \eqref{it:1a} holds.
\end{thm}

\begin{rmk}
    Adolphson--Sperber have also given a reformulation of Delaygue's criterion in terms of lattice points in polytopes \cite[Theorem 1.12 (b)]{adolphson2018integrality}, which is similar in spirit to our result, but different from it. 
    To see the difference, consider the case that $d=1$ and $(\vec{v}_j) = ((1),(1),(-1))$. In this case the data are Fano, because the origin is the unique interior lattice point of the 1-dimensional polytope $\Delta=[-1,1]$; furthermore $K_0 = \{(a,b,a+b)|a\ge 0,b \ge 0\} \cong \NN^2$, so Conjecture \ref{conj:1} \eqref{it:1a} holds in this case by Theorem \ref{thm:conj1a}. 
    On the other hand, after translating into Delaygue's setup in accordance with Section \ref{sec:translation}, Adolphson--Sperber's result says that Delaygue's criterion is equivalent to the fact that $(1,1,0,0)$ is the unique interior lattice point of the lattice polytope in $\R^4$ obtained as the convex hull of the vectors
    $$(0,0,0,0),(3,0,0,0),(0,3,0,0),(0,0,3,0),(0,0,0,3),(3,6,-3,-3).$$
    This is true, but harder to check. 
    In general, when it applies, our criterion is simpler to check than Adolphson--Sperber's (and more closely tied to the toric geometry of the mirror construction); however, their criterion applies to cases of Delaygue's result which do not arise in accordance with Section \ref{sec:translation}, so is more general.
\end{rmk}

Of course, when the rank of $K$ is greater than one, the hypothesis $K_0 \cong \NN^r$ is very much non-generic. 
For example, we have:

\begin{example}
    If $(\vec{v}_j) = ((1,0),(0,1),(-2,1),(1,-2))$ then the map
    \begin{align*}
        \Z^2 & \to K\\
        (a,b) & \mapsto (2a-b,2b-a,a,b)
        \end{align*}
        is an isomorphism. 
    Under this isomorphism, $K_0$ get identified with $\{(a,b) \in \Z^2| 2a \ge b, 2b \ge a\}$, which is not isomorphic to $\NN^2$. Indeed, if it were, then the primitive generators of the extremal rays of the cone would have to form a basis for $K_0$; but the generators are $(1,2)$ and $(2,1)$, which do not generate the element $(1,1) \in K_0$, hence are not a basis. 
    Similarly, if $(\vec{v}_j) = ((1),(1),(-1),(-1))$ then $$K_0 = \{(a,b,c,d) \in \NN^4| a+b=c+d\} \not \cong \NN^3,$$
    as the cone has four extremal rays (generated by $(1,0,1,0),(1,0,0,1),(0,1,1,0)$, and $(0,1,0,1)$).  
    See also \cite[Lemma 1.6]{ganatra2024integrality}.
\end{example}

\begin{example}
    One can easily construct examples where the rank of $K$ is greater than one but the hypothesis $K_0 \cong \NN^r$ does hold, so that Delaygue's result implies Conjecture \ref{conj:1} \eqref{it:1a} by Theorem \ref{thm:conj1a}: this is the case, e.g., for Example \ref{eg:rank2}. 
    For examples which are more interesting from the perspective of mirror symmetry, one can consider the case that $(\vec{v}_j)$ is the set of lattice points lying in the interiors of codimension-$\ge 2$ faces of a 3-dimensional reflexive polytope, of which there are 4319 examples \cite{kreuzer1998classification}. 
    Specifically, for these cases, Conjecture \ref{conj:2} implies the case of \cite[Conjecture 6.3.4]{batyrev1995generalized} concerning K3 hypersurfaces in toric Fano threefolds; and we have explained the relationship between Conjecture \ref{conj:2} and \ref{conj:1} \eqref{it:1a} above. 
    Using the code described in Appendix \ref{sec:code}, we may find many examples where the condition $K_0 \cong \NN^r$ is satisfied with $r>1$.  
    For example, the first of these (in the order in which they are listed in the PALPreader package on SageMath) is $(\vec{v}_j) = ((1,0,0),(0,1,0),(0,0,1),(-1,-1,-1),(-1,-1,1)$: it is easy to verify that $K_0 \cong \NN^2$ is generated by $(1,1,1,1,0)$ and $(2,2,0,1,1)$ in this case.
\end{example}

We know even less about Conjecture \ref{conj:1} \eqref{it:1b}: using \cite{krattenthaler2011analytic}, we prove in Section \ref{sec:Conj1b proof} that

\begin{thm}\label{thm:conj1b}
    Suppose that $K$ has rank $1$. Then $\phi_{ij}/\phi_0 \in \Q_{\ge 0}[[\vec{z}]]_{K_0}$. 
    In particular, Conjecture \ref{conj:1} \eqref{it:1b} holds (but the result is more general: it holds even if $(\vec{v}_{ij})$ is not Fano).
\end{thm}

We have no proofs of Conjecture \ref{conj:1} \eqref{it:1b} in cases where $\text{rk} (K) > 1$. 
One might ask, in light of Theorem \ref{thm:conj1b}, if the Fano hypothesis in Conjecture \ref{conj:1} \eqref{it:1b} is necessary at all. 
Plugging random examples into a computer, we found several non-Fano examples such that $\phi_{ij}/\phi_0$ has positive coefficients up to high order, but we did also find some non-Fano examples with negative coefficients. 
So it seems plausible that the Fano hypothesis could be relaxed, but it can't be completely dropped.

\begin{rmk}
    By combining \cite[Corollary 5.5.4 (ii)]{CF-K} with Givental's mirror theorem \cite{Givental_toric}, one obtains an enumerative interpretation for the coefficients of the true mirror map in terms of quasimap invariants. In cases where Lemma \ref{lem:12equiv} applies, so that the true mirror map coincides with the naive mirror map, this may be helpful for verifying Conjecture \ref{conj:1} \eqref{it:1b}. We thank Ionu\c{t} Ciocan-Fontanine for this observation.
\end{rmk}

Regarding Conjecture \ref{conj:2}, which we recall is the case of interest for mirror symmetry, we of course have proofs in cases where the hypotheses of Theorem \ref{thm:conj1a} and Lemma \ref{lem:12equiv} apply. Beyond that, there is the following result of Beukers--Vlasenko:

\begin{thm}[Corollary 7.11 of \cite{BeukersVlasenkoIII}]
    \label{thm:beukersvlasenko}
    Let $\Delta \subset \R^d$ be a reflexive polytope, whose only lattice points are the origin and the vertices, and let $G \subset \mathrm{GL}(d,\Z)$ be a group which preserves $\Delta$ and acts transitively on the vertices. 
    Let $(\vec{v}_{j})$ be the vertices of $\Delta$. 
    Then $\exp((\phi_{j} + \tau_{j})/\phi_0)(t,\ldots,t) \in \Q[[t]]$ has only finitely many primes appearing in the prime factorizations of the denominators of its coefficients.
\end{thm}

We also have the following result, proved via an arithmetic refinement of homological mirror symmetry:\footnote{We remark that the result was stated as being contingent on certain foundational results, which have since been established, so that the theorem is now proved unconditionally. Namely, \cite[Conjecture 1.14]{ganatra2015} has been verified in \cite{Tu2024}, and the conditions on the relative Fukaya category enumerated in \cite[Section 4]{ganatra2015} have been verified in \cite{relfukii,cyclicoc}.}

\begin{thm}[Theorem B of \cite{ganatra2024integrality}]\label{thm:hms}
Let $\Delta \subset \R^d$ be a reflexive simplex, and let $(\vec{v}_{j})$ be the lattice points lying on facets of $\Delta$ of codimension $2$. Suppose that $\Delta$ admits a vector satisfying the MPCS condition [\emph{op. cit.}, Definition 1.7]. Then for any $\vec{k} \in K$, the power series
$$ \prod_j \exp\left(\frac{\phi_j+\tau_j}{\phi_0}\right)^{k_j}$$
has integer coefficients.
\end{thm}

During the copy-editing process for this paper, we learned of the PhD thesis of Jorin Schug \cite{Schug}, which proves the analogue of Theorem \ref{thm:hms} for a broad class of examples relating to mirror symmetry for Calabi--Yau complete intersections in toric varieties, following a different approach proposed by Jockers--Mayr \cite{JockersMayr}.

\begin{question}
    Does the naive mirror map have any geometric significance, and can its integrality (perhaps even its positivity) be explained geometrically, as is the case for the true mirror map in Theorem \ref{thm:hms}? Does it define a mirror map relating the Yukawa coupling to some alternative curve-counting invariants (such as relative Gromov--Witten invariants \cite{LiRuan,Li,IonelParker})? For example, let us consider the case that the $(\vec{v}_j)$ are lattice points on some reflexive polytope $\Delta$. 
    Then the mirror map associated to $(\vec{v}_j)$ makes its appearance in mirror symmetry for a Calabi--Yau hypersurface in a toric variety whose fan has rays pointing along the vectors $\vec{v}_j$. 
    In particular, the vectors $\vec{v}_j$ correspond to divisors $D_j$ in this toric variety. 
    If $D_j$ is ample, then one can show that $\tau_j = 0$ by Lemma \ref{lem:12equiv}. 
    This leads one to wonder if $\exp(\tau_{j}/\phi_0)$ might be a generating function for some kind of curves living inside $D_j$, and these counts vanish when $D_j$ is ample because the dimension of the moduli space of curves inside $D_j$ is lower than that of curves in the ambient space by adjunction.
\end{question}

\textbf{Acknowledgements:} This project grew out of the first author's M.Math. dissertation, supervised by the second author. The project was to survey what was known and what was conjectured about the `integrality of mirror maps' phenomenon, and check the conjectures on a computer. However, at the start of the project N.S. mistakenly omitted the term $\tau_{ij}$ in the formula for the true mirror map, and so S.B. set about checking integrality of the coefficients of the naive mirror map. At the point when Masha Vlasenko pointed out the mistake, S.B. had already checked that the naive mirror map had positive integer coefficients in thousands of examples. This led us to the distinct Conjectures \ref{conj:1} and \ref{conj:2}. We are very grateful to Vlasenko for pointing this out, and for other helpful conversations and encouragement. The authors are also grateful to Ionu\c{t} Ciocan-Fontanine for bringing the references \cite{CF-K,Schug} to our attention, as well as other helpful remarks. N.S. is also grateful to his co-authors on the paper \cite{ganatra2024integrality}, which inspired this project; especially to Dan Pomerleano for a discussion related to the question of a geometric interpretation for the naive mirror map. 
The authors are grateful to the referees for numerous helpful suggestions and corrections.

S.B. was supported by the Royal Society through N.S.'s University Research Fellowship. N.S. was supported by ERC Starting Grant (award number 850713 -- HMS), a Royal Society University Research Fellowship, the Leverhulme Prize, and a Simons Investigator award (award number 929034).

\section{Translating between our setup and Delaygue's}\label{sec:translation}

\subsection{Delaygue's setup}

We start by recalling Delaygue's setup \cite{delaygue2013criterion}. 
Suppose we are given vectors $(\vec{e}_i)_{i=1}^p$ and $(\vec{f}_k)_{k=1}^s$ in $\NN^r$, satisfying
\begin{equation}
    \label{eq:sum e sum f}
    \sum_{i=1}^p \vec{e}_i = \sum_{k=1}^s \vec{f}_k.
\end{equation}
We will consider the special case of Delaygue's setup in which $s = |I|$, with $I$ as in \eqref{eq:I}, and
\begin{equation}
    \label{eq:e from f}
    \vec{e}_i = \sum_{j=1}^{q_i} \vec{f}_{ij},
\end{equation}
which clearly implies \eqref{eq:sum e sum f}. (We implicitly choose an ordering of $I$, so that we can relabel the $\vec{f}_{ij}$ as $\vec{f}_k$.)

Delaygue defines 
$$F_{e,f}(\vec{w}) = \sum_{\vec{n} \in \NN^r} \prod_{i=1}^p \comb(\vec{f}_{i1}\cdot \vec{n},\ldots,\vec{f}_{iq_i} \cdot \vec{n}) \cdot \vec{w}^{\vec{n}} $$
and 
$$G_{\vec{L},e,f}(\vec{w}) = \sum_{\vec{n} \in \NN^r} \prod_{i=1}^p \comb(\vec{f}_{i1}\cdot \vec{n},\ldots,\vec{f}_{iq_i} \cdot \vec{n}) \cdot H_{\vec{L} \cdot \vec{n}} \cdot \vec{w}^{\vec{n}}$$
for $\vec{L} \in \NN^r$, which are both power series in $\Q[[w_\ell]]_{\ell=1}^r$. 
His results concern integrality of the power series 
$$q_{\vec{L},e,f} = \exp(G_{\vec{L},e,f}/F_{e,f}),$$ 
for different values of $\vec{L}$.

\subsection{Translation from our setup}

Suppose that we have an isomorphism of monoids $F: \NN^r \xrightarrow{\sim} K_0$. 
Let $(\vec{f}_{ij})_{(i,j) \in I}$ be the row vectors of the matrix of $F$; so they are vectors in $\NN^r$ such that
\begin{align*}
    F(\vec{n}) &= (\vec{f}_{ij} \cdot \vec{n})_{(i,j) \in I}.
\end{align*}
We define $\vec{e}_i$ by \eqref{eq:e from f}.

We now have an isomorphism
\begin{align*}
    \iota: \Q[[w_\ell]]_{\ell=1}^r & \xrightarrow{\sim} \Q[[\vec{z}]]_{K_0}\\
    \text{sending }\qquad \vec{w}^{\vec{n}} & \mapsto \vec{z}^{F(\vec{n})}.
\end{align*}

The translation from our setup to Delaygue's is given by the following Lemma, whose proof is immediate from the definitions:

\begin{lem}\label{lem:translation}
We have
\begin{align*}
\nonumber    \phi_0 &= \iota(F_{e,f})\quad \text{and}\\
\nonumber    \phi_{ij} &= \iota(G_{\vec{e}_i,e,f} - G_{\vec{f}_{ij},e,f}).
\end{align*}
In particular, we have
\begin{align}
    \label{eq:phi_q}    \exp\left(\frac{\phi_{ij}}{\phi_0}\right) &= \iota\left(\frac{q_{\vec{e}_i,e,f}}{q_{\vec{f}_{ij},e,f}}\right).
\end{align}    
\end{lem}

\section{Proof of Theorem \ref{thm:conj1a}}\label{sec:Conj1a proof}

\subsection{Delaygue's criterion}

We now introduce Delaygue's criterion for integrality of mirror maps. It involves the function
\begin{align*}
    \Delta_{e,f}: \R^r & \to \Z,\\
    \Delta_{e,f}(\vec{x}) &\coloneqq \sum_{i=1}^p \lfloor \vec{e}_i \cdot \vec{x} \rfloor - \sum_{i=1}^p \sum_{j=1}^{q_i} \lfloor \vec{f}_{ij} \cdot \vec{x} \rfloor \\
    &= \sum_{i=1}^p \left\lfloor \sum_{j=1}^{q_i} \{\vec{f}_{ij} \cdot \vec{x}\} \right\rfloor,
\end{align*}
where $\lfloor \cdot \rfloor$ denotes the integer part, and $\{ \cdot \}$ the fractional part.

The special case of Delaygue's theorem of interest to us is:

\begin{thm}[Theorem 1.2 of \cite{delaygue2013criterion}]
    If $\Delta_{e,f}(\vec{x}) \ge 1$ for all $\vec{x} \in [0,1)^r$ such that $\vec{e}_i \cdot \vec{x} \ge 1$ for some $i$, then $q_{\vec{e}_i,e,f}$ and $q_{\vec{f}_{ij},e,f}$ have integer coefficients. As the leading coefficient of both is $1$, this implies that their quotient \eqref{eq:phi_q} has integer coefficients; so Conjecture \ref{conj:1} \eqref{it:1a} holds.
\end{thm}

Thus, Theorem \ref{thm:conj1a} follows from:

\begin{prop}\label{prop:eq_crit}
    In the setting of Section \ref{sec:translation}, $(\vec{v}_{ij})$ is Fano if and only if we have $\Delta_{e,f}(\vec{x}) \ge 1$ for all $\vec{x} \in [0,1)^r$ such that $\vec{e}_i \cdot \vec{x} \ge 1$ for some $i$.
\end{prop}

\subsection{Proof of equivalence of the criteria}\label{sec:equivalence of Fano and Delaygue}

By abuse of notation, we will also denote by
\begin{align*}
    F: \R^r & \to \R^I\\
    \vec{x} &\mapsto \left(\vec{x} \cdot \vec{f}_{ij}\right)_{(i,j) \in I}
\end{align*} 
the linear extension of the function $F$ considered in Section \ref{sec:translation}. 

\begin{lem}\label{lem:Finj}
    The map $F$ is injective.
\end{lem}
\begin{proof}
    If the linear map $F$ were not injective, we would have $\ker(F) \neq \{0\}$. 
    As the subspace $\ker(F)$ is defined over the rationals, it is non-zero if and only if it contains a rational, and hence an integral point $\vec{k} \neq \vec{0}$. 
    Writing $\vec{k} = \vec{k}^+ - \vec{k}^-$ with $\vec{k}^\pm \in \NN^r$, we obtain that $F(\vec{k}^+) = F(\vec{k}^-)$, which implies that $\vec{k}^+ = \vec{k}^-$ by our assumption that $F|_{\NN^r}$ is an isomorphism and in particular injective. Therefore $\vec{k} = \vec{0}$, a contradiction. 
\end{proof}

\begin{lem}\label{lem:extension_isosm}
    The restriction of $F$ to $\Z^r \subset \R^r$ defines an isomorphism $F|_{\Z^r}:\ZZ^r\xrightarrow{\sim}K$.
\end{lem}
\begin{proof}
    It is evident that $F$ defines a homomorphism $\Z^r \to K$; and this homomorphism is injective by Lemma \ref{lem:Finj}.
    Thus it remains to prove the surjectivity of $F$ onto $K$. 
    To this end, let $\vec{k} \in K$ be arbitrary: we will show that it lies in the image of $F$. 
    By Assumption \ref{ass}, there exists $\vec{j} \in K \cap \NN_{>0}^I$. 
    Now we may choose $\alpha \in \NN_{>0}$ sufficiently large that $\alpha \cdot \vec{j} + \vec{k} \in \NN^I$. 
    Then both $\alpha \cdot \vec{j} + \vec{k}$ and $\alpha \cdot \vec{j}$ lie in $K \cap \NN^I =: K_0$, and hence lie in the image of $F|_{\NN^r}$, by our assumption that $F$ defines an isomorphism from $\NN^r$ to $K_0$. 
    In particular, their difference $\vec{k} = (\alpha \cdot \vec{j} + \vec{k}) - \alpha \cdot \vec{j}$ lies in the image of $F|_{\Z^r}$, as required.
\end{proof}

\begin{lem}\label{lem:extension_pos}
The restriction of $F$ to $\R^r_{\ge 0} \subset \R^r$ defines an isomorphism $F|_{\R^r_{\ge 0}} : \R^r_{\ge 0} \xrightarrow{\sim} (K \otimes \R) \cap \R^I_{\ge 0}.$
\end{lem}
\begin{proof}
    By hypothesis, $F|_{\NN^r}: \NN^r \to K \cap \NN^I$ is an isomorphism. It follows by linearity of $F$ that $F|_{\Q^r_{\ge 0}} :\Q^r_{\ge 0} \to (K \otimes \Q) \cap \Q_{\ge 0}^I$ is an isomorphism. 
    The result now follows from the fact that $\Q^r_{\ge 0} \subset \R^r_{\ge 0}$ and $K \otimes \Q \cap \Q_{\ge 0}^I \subset K \otimes \R \cap \R^I_{\ge 0}$ are dense subsets (an instance of the general fact that the rational points of any rational cone are dense), together with the continuity of $F$.
\end{proof}

We have a short exact sequence of free abelian groups
\begin{equation}\label{eq:SES}
0 \to \Z^r \xrightarrow{F|_{\Z^r}} \Z^I \xrightarrow{\V} \Z^d \to 0.
\end{equation}
Exactness on the left and in the middle follows from the fact that $F$, by Lemma \ref{lem:extension_isosm}, defines an isomorphism from $\Z^r$ to $K = \ker(\V)$; exactness on the right, i.e.,  surjectivity of $\V$, is part of Assumption \ref{ass}.
We also define the function
\begin{align*}
    \{\cdot\}: \R^I & \to [0,1)^I \\
    \{(y_{ij})\} & \coloneqq \left(\{y_{ij}\}\right),
\end{align*}
and use $\{F\}:\R^r \to [0,1)^I$ to denote its composition with $F$. 
We introduce the subsets 
\begin{align*}
    \mathcal{X} & \coloneqq \left\{ \vec{x} \in [0,1)^r : \vec{x} \cdot \vec{e}_i < 1 \,\forall\, i\right\} \qquad \text{and}\\
    \mathcal{Y}&\coloneq\left\{\harpoon y\in[0,1)^I:\sum_{j=1}^{q_i}y_{ij}<1\,\forall\, i\right\}. 
\end{align*}

Delaygue's criterion ``$\Delta_{e,f}(\vec{x}) \ge 1$ for all $\vec{x} \in [0,1)^r$ such that $\vec{e}_i \cdot \vec{x} \ge 1$ for some $i$'' is then manifestly equivalent to $\mathcal{X}^c \subseteq \{F\}^{-1}\left(\mathcal{Y}^c\right)$ (where complements are taken within $[0,1)^r$, respectively $[0,1)^I$), which in turn is equivalent to $\{F\}^{-1}(\mathcal{Y}) \cap [0,1)^r \subseteq \mathcal{X}$.

On the other hand, note that by definition of convex hull and Minkowski sum, we have
\begin{align*}
    \Delta_i & = \left\{\sum_j y_{ij} \vec{v}_{ij}: y_{ij} \ge 0, \sum_j y_{ij} \le 1\right\}, \qquad \text{and}\\
    \Delta &= \V\left(\overline{\mathcal{Y}}\right),
\end{align*}
where $\overline{\mathcal{Y}}$ is the closure of $\mathcal{Y}$ in $[0,1]^I$. (To explain why the first line has $\le 1$ rather than $=1$, recall that by definition $\Delta_i$ is the convex hull of the vectors $\vec{v}_{ij}$ together with $\vec{0}$.)

Having expressed Delaygue's criterion in terms of $\{F\}$, $\mathcal{X}$, and $\mathcal{Y}$, and our polytope $\Delta$ in terms of $\V$ and $\overline{\mathcal{Y}}$, we now proceed to address the existence of interior lattice points of $\Delta$ in the same terms. 

\begin{lem}\label{lem:latt_points}
    The point $\vec{q} \in \R^d$ is an interior lattice point of $\Delta$ if and only if $\vec{q} = \V(\vec{y})$ where
    $$\vec{y} = \{F\}(\vec{x}) \in \mathcal{Y}$$
    for some $\vec{x} \in [0,1)^r$.
\end{lem}
\begin{proof}
    Let $\mathcal{Y}^\circ := \mathcal{Y} \cap (0,1)^I$. 
    We start by observing the following:
    \begin{align}\label{eq:firstequiv}
        \vec{q} \in \Delta^\circ &\iff \vec{q} = \V(\vec{y}) \text{ for some $\vec{y} \in \mathcal{Y}^\circ$}\\
       \nonumber &\iff \vec{q} = \V(\vec{y}) \text{ for some $\vec{y} \in \mathcal{Y}$.}
    \end{align}
    The first line follows from \cite[Theorem 6.6]{Rockafellar}, which says that the image of the relative interior of a convex set under a linear map is the relative interior of the image. Applying this to the convex set $\overline{\mathcal{Y}}$ (which is full-dimensional) and the linear map $\V$ (which is surjective), we obtain
    $$ \V\left(\mathcal{Y}^\circ\right)  =  \V\left(\overline{\mathcal{Y}}\right)^\circ = \Delta^\circ,$$
    which is equivalent to the statement. 
    For the second line, the $\Rightarrow$ implication is clear as $\mathcal{Y}^\circ \subset \mathcal{Y}$, so it remains to prove the $\Leftarrow$ implication. 
    Suppose that $\vec{q} = \V(\vec{y})$ with $\vec{y} \in \mathcal{Y}$, then $\vec{q} = \V(\vec{y} + \epsilon \vec{k})$ for any $\epsilon \in \R$, $\vec{k} \in K$; by taking $\vec{k} \in K \cap \NN_{>0}^I$ (which exists by Assumption \ref{ass}) and $\epsilon>0$ very small, we may arrange that $\vec{y} + \epsilon \vec{k} \in \mathcal{Y}^\circ$.

    The proof will be completed by proving the following chain of equivalences:
    \begin{align}\label{eq:secondequiv}
        \vec{q} \in \Delta^\circ \cap \Z^d & \iff \vec{q} = \V(\vec{y}) = \V(\vec{p}) \text{ for some $\vec{y} \in \mathcal{Y}$, $\vec{p} \in \Z^I$}\\
        \nonumber&\iff \vec{q} = \V(\vec{y}) \text{ for some $\vec{y} \in \mathcal{Y}$, and } \vec{y} = \{F\}(\vec{x}) \text{ for some $\vec{x} \in \R^r$}\\
        \nonumber&\iff \vec{q} = \V(\vec{y}) \text{ for some $\vec{y} \in \mathcal{Y}$, and } \vec{y} = \{F\}(\vec{x}) \text{ for some $\vec{x} \in [0,1)^r$.}
    \end{align}
    The first line follows from the equivalence \eqref{eq:firstequiv} proved above, combined with the surjectivity of the map $\V$. 
    For the implication $\Rightarrow$ in the second line, we use the exactness of \eqref{eq:SES} (tensored with $\R$) to observe that $\V(\vec{y}) = \V(\vec{p})$ if and only if $\vec{y} - \vec{p} = F(\vec{x})$ for some $\vec{x} \in \R^r$; as $\vec{y} \in [0,1)^I$, this implies $\vec{y} = \{F\}(\vec{x})$. 
    For the implication $\Leftarrow$ in the second line, we set $\vec{p} = \vec{y} - F(\vec{x})$, and observe that $\vec{p} = \{F\}(\vec{x}) - F(\vec{x}) \in \Z^I$. 
    For the third line, we observe that $\{F\}(\vec{x} + \vec{a}) = \{F\}(\vec{x})$ for any $\vec{a} \in \Z^r$, so we may choose $\vec{x} \in [0,1)^r$ without loss of generality.
\end{proof}

\begin{lem}\label{lem:0 latt point}
    In the setting of Lemma \ref{lem:latt_points}, we have $\vec{q} = \vec{0}$ if and only if $\vec{x} \in \mathcal{X}$.
\end{lem}
\begin{proof}
    To prove the $\Leftarrow$ implication, notice that if $\vec{x}\in \mathcal{X}$ then $\vec{x}\cdot\vec{e}_i<1$ for all $i$ and $\vec{x}\in [0,1)^r$. Then $\vec{x}\cdot \vec{f}_{ij}<1$ for all $i,j$, and $\{\vec{x}\cdot \vec{f}_{ij}\}=\vec{x}\cdot \vec{f}_{ij}$, so that $\{F\}(\vec{x})=F(\vec{x})$, and
    \[\vec{q} = \V(\vec{y})=\V( F(\vec{x}))=\vec{0},\]
    as $\V \circ F = 0$.
    
    To prove the $\Rightarrow$ implication, notice that if $\vec{q} = \vec{0}$ then $\vec{y} = F(\vec{a})$ for some $\vec{a} \in \R^r$, by the exactness of \eqref{eq:SES}. 
    As $\vec{y} \in \mathcal{Y} \subset \R_{\ge 0}^I$, we must have $\vec{a} \in \R_{\ge 0}^r$ by Lemma \ref{lem:extension_pos}. 
    As $F(\vec{a}) \in \mathcal{Y} \subset [0,1)^I$, and the matrix of $F$ has non-negative integer entries and is of full rank, it follows that $\vec{a} \in [0,1)^r$. 
    As $\{F\}(\vec{x}) = F(\vec{a})$, we have $F(\vec{x}-\vec{a}) \in \Z^I$, from which it follows by exactness of \eqref{eq:SES} that $\vec{x} - \vec{a} \in \Z^r$. 
    As both $\vec{x}$ and $\vec{a}$ lie in $[0,1)^r$, it follows that $\vec{x} = \vec{a}$.

    Thus, we have $F(\vec{x}) = \{F\}(\vec{x}) \in \mathcal{Y}$, with $\vec{x} \in [0,1)^r$, from which it follows that $\vec{x} \in \mathcal{X}$. 
\end{proof}

\begin{proof}[Proof of Proposition \ref{prop:eq_crit}]
    Putting together Lemmas \ref{lem:latt_points} and \ref{lem:0 latt point}, $\vec{0}$ is the unique interior lattice point of $\Delta$ if and only if $\{F\}^{-1}(\mathcal{Y}) \cap [0,1)^r \subset \mathcal{X}$, which we have established is equivalent to Delaygue's criterion.
\end{proof}

\section{Proof of Theorem \ref{thm:conj1b}}\label{sec:Conj1b proof}

Recall that Conjecture \ref{conj:1} \eqref{it:1b} proposed that if $(\vec{v}_{ij})$ is Fano, then for all $(i,j) \in I$, $\log \psi^{\text{n}}_{ij} = \phi_{ij}/\phi_0$ has non-negative coefficients; i.e., it lies in $\Q_{\ge 0}[[\vec{z}]]_{K_0}$. Theorem \ref{thm:conj1b} states that this holds when the rank of $K$ is $1$; we now turn to its proof.

If $K$ has rank $1$, then by Assumption \ref{ass}, we have $K \cap \Z_{>0}^I \neq \emptyset$, from which it follows that $K_0 \cong \NN$. 
Thus we may translate to Delaygue's setup in accordance with Section \ref{sec:translation}. 
As $r=1$, the vectors $\vec{e}_i$ and $\vec{f}_{ij}$ are in fact natural numbers $e_i$ and $f_{ij}$. 
By Lemma \ref{lem:translation}, $\phi_{ij}/\phi_0$ has non-negative coefficients if and only if $b(w)/a(w)$ has, where
$$
    a(w) = \sum_{n =0}^\infty a_n \cdot w^n,\qquad b(w)=\sum_{n=1}^\infty a_n \cdot c_n \cdot w^n,
$$
where we have set 
$$a_n = \prod_{\ell=1}^p \comb(nf_{\ell 1},\ldots,nf_{\ell q_\ell}),\qquad c_n = H_{ne_i} - H_{nf_{ij}}.$$
(We fix $i$ and $j$ for the purposes of the proof.)

In order to show that $b/a$ has non-negative coefficients, it suffices by \cite[Lemmas 2.1 and 2.2]{krattenthaler2011analytic} to prove:
\begin{enumerate}
    \item \label{item:phi_0} $a_0=1$;
    \item \label{item:phi_1} $a_1> 0$;
    \item \label{item:convex} $a_n^2\leq a_{n-1}a_{n+1}$ for all $n\ge 1$ (i.e., $a_n$ is log-convex); and
    \item \label{item:increasing} $0\leq c_{n}\leq c_{n+1}$ \quad (i.e., $c_n$ is nonnegative and increasing in $n$).
\end{enumerate}
To prove item \ref{item:phi_0}, we simply observe that $\comb(\vec{0}) = 1$.
To prove item \ref{item:phi_1}, we simply observe that $\comb(\vec{k})>0$ for any $\vec{k} \in \NN^{q_i}$.

Thus it remains to prove item \ref{item:convex} and item \ref{item:increasing}. 
item \ref{item:convex} follows from Lemma \ref{lem:Phi_convex} below, and item \ref{item:increasing} follows from Lemma \ref{lem:H_k increasing} below, so the proof of Theorem \ref{thm:conj1b} is complete.

\begin{lem}[Proposition 8.1 of \cite{alexandersson2019cone}]\label{lem:An_increasing}
    For $\sigma$ a convex and decreasing function on $[0,1]$, the function
    \[A_s\coloneq \frac{\sigma\left(\frac{1}{s}\right)+\sigma\left(\frac{2}{s}\right)+\dots+\sigma\left(\frac{s}{s}\right)}{s}\]
    is increasing in $s$.
\end{lem}

\begin{lem}\label{lem:Phi_convex}
    The sequence $a_n$ is log-convex: $a_n^2\leq a_{n-1}a_{n+1}$.
\end{lem}
\begin{proof}
    Using the definition $a_n=\frac{\prod_{\ell}(ne_{\ell})!}{\prod_{\ell}\prod_{m}(nf_{\ell m})!}$, we can state the desired inequality as follows:
    \[\left(\frac{\prod_{\ell}(ne_{\ell})!}{\prod_{\ell}\prod_{m}(nf_{\ell m})!}\right)^2\leq \left(\frac{\prod_{\ell}((n-1)e_{\ell})!}{\prod_{\ell}\prod_{m}((n-1)f_{\ell m})!}\right)\left(\frac{\prod_{\ell}((n+1)e_{\ell})!}{\prod_{\ell}\prod_{m}((n+1)f_{\ell m})!}\right).\]
    We will prove the stronger statement that for any $\ell$,
    \begin{align*}
   \notag     \left(\frac{(ne_{\ell})!}{\prod_{m}(nf_{\ell m})!}\right)^2 &\leq \left(\frac{((n-1)e_{\ell})!}{\prod_{m}((n-1)f_{\ell m})!}\right)\left(\frac{((n+1)e_{\ell})!}{\prod_{m}((n+1)f_{\ell m})!}\right),
   \end{align*}
   which is equivalent (by cancelling terms and rearranging) to:
   \begin{align}
   \label{eq:a_n convex}     \frac{\prod_{d=1}^{e_{\ell}}(n-1)e_{\ell}+d}{\prod_{d=1}^{e_{\ell}}ne_{\ell}+d} &\leq \frac{\prod_m\prod_{c=1}^{f_{\ell m}}(n-1)f_{\ell m}+c}{\prod_m\prod_{c=1}^{f_{\ell m}}nf_{\ell m}+c}.
    \end{align}
    Therefore it suffices to prove equation \eqref{eq:a_n convex}.
    
    To this end, let $\sigma:[0,1]\to\R_{\geq 1}$ be the function 
    \[\sigma(x)=\frac{n+x}{n-1+x}.\]
    Then $\sigma$ is decreasing, and so $\log\sigma$ is decreasing. 
    We claim that $\log\sigma$ is also convex. To see this, observe that
    $\log\sigma(x)=\log(n+x)-\log(n-1+x)$. Then 
     \[\left(\log\sigma(x)\right)''=\frac{1}{(n-1+x)^2}-\frac{1}{(n+x)^2},\]
     which is positive as $n+x>n-1+x$.
        
    By Lemma \ref{lem:An_increasing}, we then have that
    \[\frac{\log\sigma\left(\frac{1}{s}\right)+\dots+\log\sigma\left(\frac{s}{s}\right)}{s}\leq \frac{\log\sigma\left(\frac{1}{s+1}\right)+\dots+\log\sigma\left(\frac{s+1}{s+1}\right)}{s+1},\]
    so that 
    \[\left(\sigma\left(\frac{1}{s}\right)\cdot\dots\cdot\sigma\left(\frac{s}{s}\right)\right)^{\frac{1}{s}} \leq 
    \left(\sigma\left(\frac{1}{s+1}\right)\cdot\dots\cdot\sigma\left(\frac{s+1}{s+1}\right)\right)^{\frac{1}{s+1}}.\]
    Since $e_{\ell}\ge f_{\ell m}$ for all $1\leq m\leq q_{\ell}$, we have that
    \begin{align*}
        \left(\sigma\left(\frac{1}{f_{\ell m}}\right)\cdot\dots\cdot\sigma\left(\frac{f_{\ell m}}{f_{\ell m}}\right)\right)^{\frac{1}{f_{\ell m}}} &\leq \left(\sigma\left(\frac{1}{e_{\ell}}\right)\cdot\dots\cdot\sigma\left(\frac{e_{\ell}}{e_{\ell}}\right)\right)^{\frac{1}{e_{\ell}}}\\
    \Rightarrow    \sigma\left(\frac{1}{f_{\ell m}}\right)\cdot\dots\cdot\sigma\left(\frac{f_{\ell m}}{f_{\ell m}}\right)
    &\leq  \left(\sigma\left(\frac{1}{e_{\ell}}\right)\cdot\dots\cdot\sigma\left(\frac{e_{\ell}}{e_{\ell}}\right)\right)^{\frac{f_{\ell m}}{e_{\ell}}}.
    \end{align*}
    Taking the product over $m$, the exponent of the RHS becomes $\sum_m\frac{f_{\ell m}}{e_{\ell}}=1$. That is to say, we obtain 
    \[\prod_m \sigma\left(\frac{1}{f_{\ell m}}\right)\cdot\dots\cdot\sigma\left(\frac{f_{\ell m}}{f_{\ell m}}\right)\leq \sigma\left(\frac{1}{e_{\ell}}\right)\cdot\dots\cdot\sigma\left(\frac{e_{\ell}}{e_{\ell}}\right),\]
    which we can expand to get 
    \[\prod_m\prod_{c=1}^{f_{\ell m}}\frac{nf_{\ell m}+c}{(n-1)f_{\ell m}+c}\leq \prod_{d=1}^{e_{\ell}}\frac{ne_{\ell}+d}{(n-1)e_{\ell}+d},\]
    which is equivalent to the required inequality \eqref{eq:a_n convex}.
\end{proof}

\begin{lem}[Proposition 8.2 of \cite{alexandersson2019cone}]\label{lem:H_k increasing}
    The function $c_{n}$ given by
    \[c_{n}\coloneq \sum_{\ell=1}^{ne_i}\frac{1}{\ell}-\sum_{\ell=1}^{nf_{ij}}\frac{1}{\ell}\]
    is nonnegative and increasing in $n$.
\end{lem}
\begin{proof}
    Nonnegativity follows from the fact that
    $$c_n = \sum_{\ell=nf_{ij}+1}^{ne_i} \frac{1}{\ell} \ge 0,$$
    as $e_i \ge f_{ij}$.
    
    Proving that $c_{n+1}\geq c_{n}$
    is equivalent to proving that 
    \[\sum_{\ell=1}^{(n+1)e_i}\frac{1}{\ell}-\sum_{\ell=1}^{(n+1)f_{ij}}\frac{1}{\ell}\geq 
    \sum_{\ell=1}^{ne_i}\frac{1}{\ell}-\sum_{\ell=1}^{nf_{ij}}\frac{1}{\ell}.\]
    By \cite[Corollary 8.2]{alexandersson2019cone}, we have that 
    \[\frac{1}{m}\sum_{\ell=1}^{m}\frac{1}{n+\frac{\ell}{m}}\]
    is increasing in $m$. Therefore, as $e_i\geq f_{ij}$ for all $j$,
    \[\frac{1}{e_i}\sum_{\ell=1}^{e_i}\frac{1}{n+\frac{\ell}{e_i}}\geq \frac{1}{f_{ij}}\sum_{\ell=1}^{f_{ij}}\frac{1}{n+\frac{\ell}{f_{ij}}}.\]
    Then we can bring in the constants to get
    \[\sum_{\ell=1}^{e_i}\frac{1}{ne_i+\ell}\geq\sum_{\ell=1}^{f_{ij}}\frac{1}{nf_{ij}+\ell},\]
    and reindexing this gives us
    \[\sum_{\ell=ne_i+1}^{(n+1)e_i}\frac{1}{\ell}\geq \sum_{\ell=nf_{ij}+1}^{(n+1)f_{ij}}\frac{1}{\ell},\]
    which is equivalent to
    \[\sum_{\ell=1}^{(n+1)e_i}\frac{1}{\ell}-\sum_{\ell=1}^{(n+1)f_{ij}}\frac{1}{\ell}\geq 
    \sum_{\ell=1}^{ne_i}\frac{1}{\ell}-\sum_{\ell=1}^{nf_{ij}}\frac{1}{\ell},\]
    proving the Lemma.
    
\end{proof}

\appendix

\section{Computer checks}\label{sec:code}

We list the checks of Conjectures \ref{conj:1} and \ref{conj:2} we have performed on a computer. Our program, written using SageMath \cite{sagemath}, did the following. Define $\vec{1} \in \Z^I$ to be the vector all of whose entries are $1$; and $\vec{1}_{ij} \in \Z^I$ to be the vector whose $(i,j)$ entry is $1$, and all other entries are $0$. Given a `precision' parameter $P$, the program found $d$ such that 
$$K_0(d) = \# \{ \vec{k} \in K_0| \vec{k} \cdot \vec{1} \le d\} \ge P;$$
and $d'$ such that
$$K_{0,ij}(d') = \# \{ \vec{k} \in K_0 + K_{ij}| \vec{k} \cdot (2\cdot\vec{1} - \vec{1}_{ij}) \le d'\} \ge P.$$
It is clear that $K_0(d)$ is always finite; it is also true that $K_{0,ij}(d')$ is always finite, by the proof of Lemma \ref{lem:cone_conv}. 

We say `Conjecture \ref{conj:1} \eqref{it:1a} holds for the first $P$ terms' if the coefficient of $\vec{z}^{\vec{k}}$ in $\psi_{ij}^{\text{n}}$ is an integer for all $\vec{k} \in K_0(d)$.
We say `Conjecture \ref{conj:1} \eqref{it:1b} holds for the first $P$ terms' if the coefficient of $\vec{z}^{\vec{k}}$ in $\phi_{ij}/\phi_0$ is positive for all $\vec{k} \in K_0(d)$. 
We say `Conjecture \ref{conj:2} holds for the first $P$ terms' if the coefficient of $\vec{z}^{\vec{k}}$ in $\exp(\tau_{ij}/\phi_0)$ is an integer for all $\vec{k} \in K_{0,ij}(d')$. Note that the conjectures `$\psi_{ij}^{\text{n}} = \exp(\phi_{ij}/\phi_0)$ and $\psi_{ij}^{\text{n}} = \exp((\phi_{ij}+\tau_{ij})/\phi_0)$ have integer coefficients' are together equivalent to the conjectures `$\psi_{ij}^{\text{n}} = \exp(\phi_{ij}/\phi_0)$ and $\exp(\tau_{ij}/\phi_0)$ have integer coefficients'; so we test the latter pair of conjectures, as it is typically faster to compute $\exp(\tau_{ij}/\phi_0)$ (as it typically has fewer terms) than $\exp((\phi_{ij}+\tau_{ij})/\phi_0)$.

\begin{figure}
    \centering
    \includegraphics[width=0.7\linewidth]{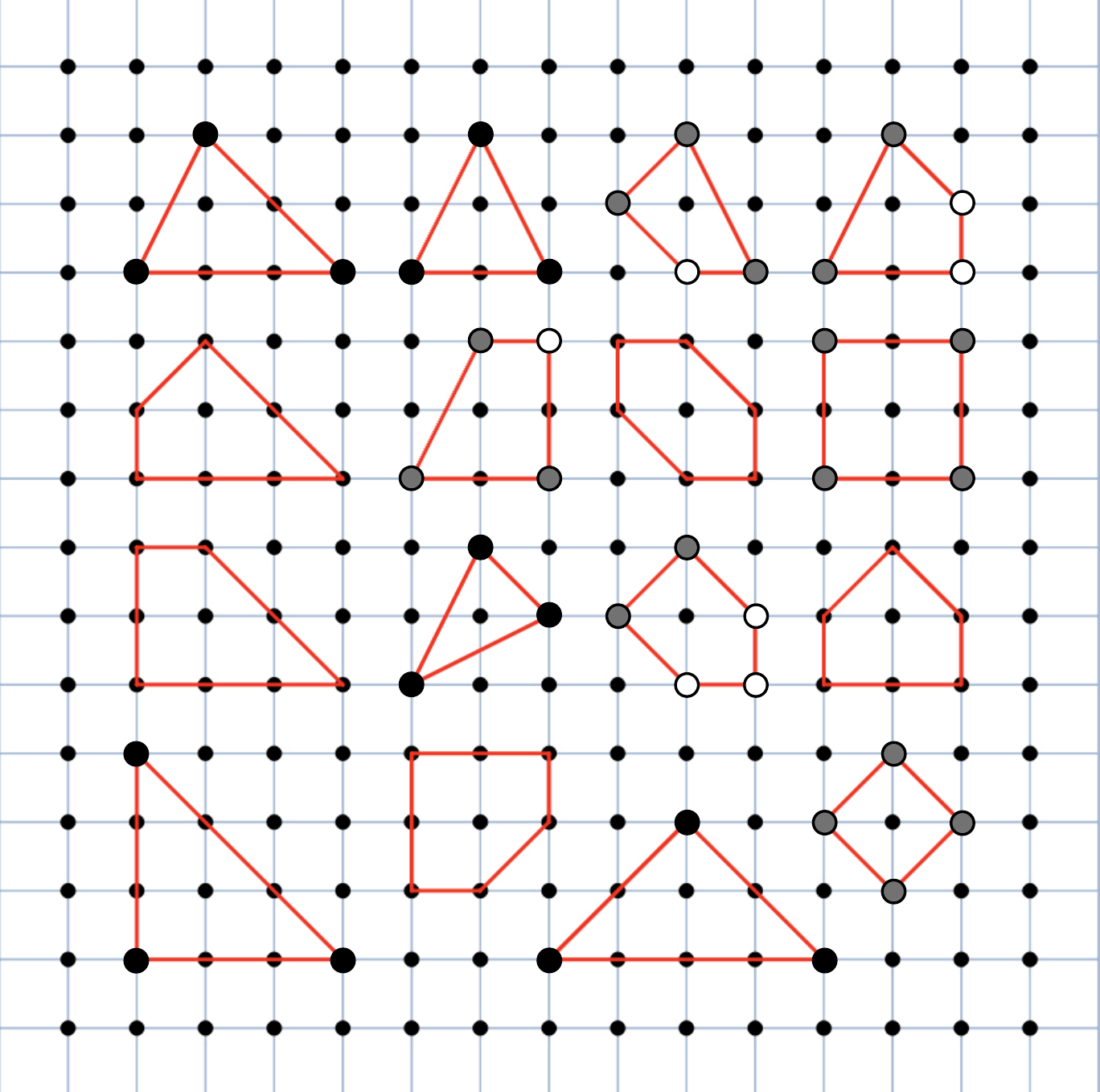}
    \caption{The set of 2-dimensional reflexive polytopes, up to unimodular equivalence. For each polytope, we consider the Fano data $(\vec{v}_j)$ given by the vertices of the polytope. For each vertex $\vec{v}_j$ of a polytope, we indicate the status of Conjecture \ref{conj:1} \eqref{it:1a} (i.e., whether $\psi_j^{\text{n}}$ has been proven to have integer coefficients, by Theorem \ref{thm:conj1a}), Conjecture \ref{conj:1} \eqref{it:1b} (i.e., whether $\log \psi_j^{\text{n}}$ has been proven to have positive coefficients, by Theorem \ref{thm:conj1b}), and Conjecture \ref{conj:2} (i.e., whether $\psi_j^{\text{t}}$ has been proven to have integer coefficients, by showing that it is equal to $\psi_j^{\text{n}}$ using Lemma \ref{lem:12equiv} and applying Theorem \ref{thm:conj1a}). 
    We do this by placing a black dot on the vertex if Conjectures \ref{conj:1} \eqref{it:1a}, \ref{conj:1} \eqref{it:1b}, and \ref{conj:2} are proven; a grey dot on the vertex if only Conjectures \ref{conj:1} \eqref{it:1a} and \ref{conj:2} are proven; a white dot if only Conjecture \ref{conj:1} \eqref{it:1a} is proven; and no dot if all three conjectures are open. }
    \label{fig:2_diml_polytopes}
\end{figure}

Let $(\vec{v}_j)$ be the vertices of a $2$-dimensional reflexive polytope, of which there are 16 up to unimodular equivalence, as shown in Figure \ref{fig:2_diml_polytopes}. 
We indicate the cases of Conjectures \ref{conj:1} and \ref{conj:2} which are proven in the Figure. 
    We checked the remaining cases on a computer, and found that they hold for the first $50$ terms. 
Note that this is not the only way of generating Fano data $(\vec{v}_j)$ from the $2$-dimensional reflexive polytopes: we may take lattice points other than the vertices, and we may take each lattice point multiple times. However we have not performed further checks in such cases.

    

Let $(\vec{v}_j)$ be the vertices of a $3$-dimensional reflexive polytope, of which there are 4319 \cite{kreuzer1998classification}. Conjecture \ref{conj:1} \eqref{it:1a} holds by Theorem \ref{thm:conj1a} for 825 of these examples, and Conjecture \ref{conj:1} \eqref{it:1b} holds by Theorem \ref{thm:conj1b} for 48 of these examples. 
We checked that Conjectures \ref{conj:1} and \ref{conj:2} hold for the first $50$ terms, for all but 23 cases which turned out to be especially computationally intensive (because $K$ has low rank, so the lattice points in $K_0(d)$ have relatively large coefficients); we checked that Conjectures \ref{conj:1} and \ref{conj:2} hold for the first 25 terms for these 23 cases.

Let $(\vec{v}_j)$ be the lattice points on the edges of a $3$-dimensional reflexive polytope. We carried out the following check of Conjecture \ref{conj:1}: the coefficient of $\vec{z}^{\vec{k}}$ in $\psi_{j}^{\mathrm{n}}$ is an integer, and the coefficient of $\vec{z}^{\vec{k}}$ in $\log(\psi_{j}^{\mathrm{n}})$ is positive, for all $\vec{k} \in K_0(d_{\mathrm{rank}(K)})$, where 
\[(d_1,d_2,\dots,d_{22})=(30,30,30,30,30,20,10,9,8,7,6,5,4,4,3,3,3,3,3,3,3,3).\]
(This check was carried out as part of the first author's M.Math. dissertation, and was checked up to a given value of $d$ rather than up to a given number of terms.)

Although it is impractical to enumerate Fano or reflexive polytopes in higher dimensions, we generated some examples in an ad-hoc way as follows. We took the $5$-dimensional reflexive polytope $\Delta$ with vertices $6\vec{1}_i - \vec{1}$ for $i=1,\ldots,5$ (where recall $\vec{1}_i$ are the standard basis vectors and $\vec{1}$ is their sum), together with $-\vec{1}$, and chose $(\vec{v}_j)$ to be $10$ random lattice points in $\partial \Delta$. We discarded the resulting data $(\vec{v}_j)$ if it was not Fano, or if the convex hull of the $\vec{v}_j$ was reflexive, or if Theorem \ref{thm:conj1a} applied to it (Theorem \ref{thm:conj1b} never applies as the rank of $K$ for these examples is $5$). 
We generated 20 examples in this way, and checked that Conjectures \ref{conj:1} and \ref{conj:2} held for the first $50$ terms in these examples.

In order to generate some examples with $p >1$, we took $\nabla$ to be the octahedron (i.e., the convex hull of vectors $\pm \vec{1}_1,\pm \vec{1}_2,\pm \vec{1}_3$ where $\vec{1}_i$ are the standard basis vectors of $\Z^3$). Any partition of the vertices of $\nabla$ defines a nef partition; choosing a 2-part or 3-part partition at random, we took the dual $(\vec{v}_{ij})$ to this random nef partition in accordance with \cite{Batyrev_Borisov}, and let $\Delta_i$ be the lattice polytopes associated to $(\vec{v}_{ij})$. 
We then chose a random lattice point in two of the $\Delta_i$, and added them to the respective list $\vec{v}_{ij}$. Note that this does not change the convex hull of the $\Delta_i$, so the data remain Fano. The reason we took the dual nef partition was because if we hadn't, then the only lattice points we could have added would have been additional copies of the vertices of $\Delta$; this would still be a non-trivial new case of the conjecture to check, but doesn't seem as exotic. 
We generated 15 examples with $p=2$ in this way, and $20$ examples with $p=3$. 
We checked that Conjectures \ref{conj:1} and \ref{conj:2} held for the first 80 terms in these examples.

\renewbibmacro{in:}{}
\def\bibrangedash{ -- }
\printbibliography 

\end{document}